\documentclass{iopjournal}
\usepackage{amsmath,amssymb,amsthm,bm}
\numberwithin{equation}{section}

\newtheorem{theorem}{Theorem}[section]
\newtheorem{lemma}[theorem]{Lemma}

\newtheorem{remark}[theorem]{Remark}

\newcommand{\D}{\mathbb D}

\usepackage{graphicx}
\usepackage{subcaption}

\DeclareMathOperator*{\argmax}{arg\,max}

\begin{document}


\title{A Moment--Hankel Rank Method for Identifying the Number of Point Sources in the Heat Equation}

\author{Zhiliang Deng$^1$\orcid{0000-0002-3319-3385}, Xiaomei Yang$^{2, *}$ and Ailin Qian$^{3}$}

\affil{$^1$School of Mathematical Science, University of Electronic Science and Technology of China, Chengdu, China}

\affil{$^2$School of Mathematics, Southwest Jiaotong University, Chengdu, China}

\affil{$^3$Department of Mathematics and Statistics, Hubei  University of Science and Technology, Xianning, Hubei}

\email{yangxiaomaht@swjtu.edu.cn}


\begin{abstract}
We develop a low-frequency moment--Hankel rank method for identifying the
number of time-independent point sources in the heat equation from boundary
flux data. The method is formulated in the unit disk, where the
Laplace-transformed and normalized boundary flux admits an explicit Fourier
moment representation. By taking the low-frequency limit, we obtain a finite exponential-sum moment
sequence in which the nodes encode the source locations, the weights encode the
source strengths, and the number of terms equals the number of point sources.
 The associated Hankel matrix admits a Vandermonde
factorization, and its rank is exactly the source number under the natural
assumptions that the source locations are distinct and the source strengths
are nonzero.
We also analyze the effect of discrete and noisy boundary data. A uniform
moment perturbation bound is propagated to the empirical Hankel matrix, and
Weyl's singular-value perturbation inequality yields a sufficient condition
for stable numerical rank recovery in terms of the smallest nonzero singular
value of the ideal Hankel matrix. Numerical experiments confirm the exact rank
pattern in the noiseless case, validate the stability threshold under moment
noise, and illustrate the loss of resolution for close or weak sources. After
the source number is identified, the same moment sequence can be used for
location and strength recovery through an annihilating-polynomial and
Vandermonde reconstruction procedure.
\end{abstract}

\section{Introduction}
\label{sec:introduction}

%

In this paper, we consider a point source recovery problem governed by the heat equation in the unit disk \(\mathbb{D}\subset\mathbb R^2\)
\begin{align}\label{heat_equation}
\left\{
\begin{aligned}
&\partial_tu(x, t)-\Delta u(x, t)=\sum_{j=1}^{N}q_j\delta(x-p_j), && (x, t)\in \mathbb D\times(0, \infty),\\
&u(x, t)=0, && (x, t)\in \partial\mathbb D\times(0, \infty),\\
&u(x, 0)=0, && x\in\mathbb D.
\end{aligned}
\right.
\end{align}
where the locations \(p_j\in\mathbb D\), the strengths \(q_j\neq0\),  and the number \(N\) of sources are unknown and need to be determined from the boundary heat fluxes
\begin{align}
f(\theta,t)
=
\frac{\partial u}{\partial\nu}(e^{i\theta},t),
\quad 0\leq\theta<2\pi.
\end{align}

Inverse source problems for parabolic equations have received considerable
attention over the past decades. Classical results on uniqueness, stability,
and reconstruction have been established for a range of settings, including
different observation types, source regularities, and a priori information
\cite{Bushuyev1995,Cannon1998,Choulli1994,Isakov1990,Isakov1991,Kian2019,
Reeve1994,Rundell1980,Solovev1989}. These developments have been largely
driven by applications such as medical imaging, environmental monitoring,
groundwater contamination identification, and structural health monitoring
\cite{Badia2005,Baillet2001,Baratchart2005,Gallet2022,He2018,Kovalets2011,
Moghaddam2021,Shlomi2007}. Nevertheless, the recovery of point sources
represented by Dirac measures remains a distinct inverse problem, because the
unknown source configuration contains not only locations and strengths, but
also the number of active sources.

For heat equations with point sources, several recent works are directly
related to the present study. Ling et al. \cite{Ling2006} showed, under the
assumption that the source function is a sum of known spatial functions, that
one measurement point can identify the number of sources and three measurement
points can determine all source locations in two-dimensional domains. Gong et
al. \cite{Gong2026} proved uniqueness for recovering one point source with a
piecewise constant-in-time amplitude from boundary flux data, using
eigenfunction expansions, kernel estimates, and complex-analytic arguments. Gu
et al. \cite{Gu2025} established a uniqueness theorem for a Dirac point source
from sparse boundary measurements and proposed a least-squares reconstruction
method solved by gradient descent. Deng et al. \cite{Deng2026} developed a
Bayesian thinning algorithm to infer the number, locations, and intensities of
heat point sources when the cardinality is unknown. In
\cite{Deng2026_Hankel}, Deng et al. introduced a Hankel determinant
characteristic for determining the number of point sources and used the
argument principle to compute the associated determinant zero count.

Algebraic Hankel-type methods have also appeared in related inverse source
problems for elliptic equations. Abdelaziz et al. \cite{Abdelaziz2015} studied inverse source problems for two-dimensional
elliptic equations, including the Helmholtz equation, from Cauchy data. They
considered pointwise sources and sources supported in small subdomains,
derived reciprocity-gap type algebraic identities, and constructed
Hankel-type matrices to recover source information through rank identification
and a companion-matrix procedure. Point-source recovery problems for other
types of equations have also been studied; see, for example,
\cite{Baratchart2005,Komornik2002,Mamonov2013,Ohe2011,Ren2019,Vessella1992}.
These works show that point-source identification often admits finite-dimensional
algebraic structures, but the precise structure depends strongly on the
governing equation and on the type of boundary data.

The present paper follows a different low-frequency moment route for
stationary point sources in the heat equation. After Laplace transformation and
normalization, we take the low-frequency limit of the boundary heat flux and
compute its positive Fourier moments. These moments form a finite
exponential-sum sequence: the nodes encode the source locations, the weights
encode the source strengths, and the number of exponential terms is the number
of point sources. The associated Hankel matrix admits a Vandermonde
factorization, and its rank equals the source number when the source locations
are distinct and the strengths are nonzero.

This rank-based mechanism is different from the determinant zero-counting
approach in \cite{Deng2026_Hankel}. There the source number is encoded in the
zero order of an analytic Hankel determinant characteristic
\(\Delta_m(s)\), and the computation is implemented by an argument-principle
contour integral in the complex Laplace-frequency plane. Here we use only the
low-frequency moment sequence and obtain a static Hankel matrix. Therefore the
source number is identified from the rank, or numerical rank, of this matrix.
For discrete and noisy boundary data, we further quantify the perturbation of
the singular values by Weyl's inequality and derive a sufficient condition for
stable numerical source counting. Once the source number is determined, the
same moment sequence can be used to recover source locations and strengths
through an annihilating-polynomial and Vandermonde reconstruction procedure.

The rest of the paper is organized as follows. The remainder of the
Introduction formulates the inverse heat point-source problem and introduces
the normalized boundary flux. Section~\ref{sec:hankel-rank} derives the
low-frequency moment representation and proves the moment--Hankel rank
characterization of the source number. Section~\ref{sec:discrete-noisy-rank}
discusses discrete boundary measurements, noisy moments, and numerical rank
selection. Section~\ref{sec:numerical} presents numerical experiments, and
Section~\ref{sec:conclusion} concludes the paper.

\section{Moment--Hankel Rank Principle}
\label{sec:hankel-rank}

This section derives the moment--Hankel rank principle for source counting.
Starting from the Laplace-transformed boundary flux, we first remove the known
factor caused by the time-independent source strength and introduce the
normalized boundary flux. We then take its low-frequency limit and show that
the resulting positive Fourier moments form a finite exponential sum. This
finite-dimensional structure yields a Hankel matrix with a Vandermonde
factorization. As a consequence, the rank of the Hankel matrix equals the
number of point sources under the natural assumptions that the source locations
are distinct and the source strengths are nonzero.


The Laplace transform (with respect to $t$) of system \eqref{heat_equation}, using the initial condition $u(x, 0)=0$,  yields the following boundary value problem for the modified Helmholtz equation:
\begin{align}
\left\{
\begin{aligned}
&s\widehat u(x, s)-\Delta\widehat u(x, s)=\frac1s\sum_{j=1}^{N}q_j\delta(x-p_j),
\\
& \widehat u|_{\partial\mathbb D}=0.
\end{aligned}
\right.
\end{align}
Let $G_s(x, y)$ denote the Green's function for the modified Helmholtz operator $s-\Delta$ on $\D$ with homogeneous Dirichlet boundary conditions. Then the solution $\widehat u$ can be expressed as
\begin{align}
\widehat u(x, s)=\frac1s \sum_{j=1}^{N}q_j  G_s(e^{i\theta}, p_j).
\end{align}
Therefore   the Laplace transform of the boundary flux  takes the form
\begin{align}
\widehat f(\theta,s)=\frac1s
\sum_{j=1}^{N}q_j \frac{\partial G_s}{\partial\nu_z}(e^{i\theta}, p_j).
\end{align}
Since the factor \(1/s\) is known, we use the normalized boundary flux
\begin{align}
\mathcal F(\theta,s):=s\widehat f(\theta,s).
\end{align}
The proposed source-counting method is based on the low-frequency limit of
\(\mathcal F(\theta,s)\).

Let
\begin{align}
g(\theta):=\lim_{s\to0}\mathcal F(\theta,s).
\end{align}
Since \(G_s\to G_0\) as \(s\to0\), where \(G_0\) is the Dirichlet Green function
for \(-\Delta\) in the unit disk, we have
\begin{align}\label{sec2:eqg}
g(\theta)=\sum_{j=1}^{N}q_j \frac{\partial G_0}{\partial\nu_z}(e^{i\theta}, p_j).
\end{align}
Writing $p_j=\rho_j e^{i\phi_j}$, $0\leq\rho_j<1$,
the boundary normal derivative of \(G_0\) admits the Fourier expansion
\begin{align}\label{sec2:eqgr}
\frac{\partial G_0}{\partial\nu_z}(e^{i\theta},p_j)=-\frac1{2\pi}\sum_{n=-\infty}^{\infty}\rho_j^{|n|}e^{in(\theta-\phi_j)}.
\end{align}
Define the positive Fourier moments
\begin{align}
\mathcal M_n=
\frac1{2\pi}
\int_0^{2\pi}
g(\theta)e^{-in\theta}\,d\theta,
\qquad n=0, 1, \ldots.
\end{align}
According to \eqref{sec2:eqg} and \eqref{sec2:eqgr}, we have the representation
\begin{align}
\mathcal M_n
=-\frac1{2\pi}\sum_{j=1}^{N}q_j\left(\rho_j e^{-i\phi_j}\right)^n.
\end{align}
With
\[
\lambda_j:=\rho_j e^{-i\phi_j},
\qquad
w_j:=-\frac{q_j}{2\pi},
\]
the moments take the finite exponential-sum form
\begin{align}\label{sec:eqf}
\mathcal M_n=\sum_{j=1}^{N}w_j\lambda_j^n,
\qquad n=0, 1, \ldots.
\end{align}
Thus the source configuration is encoded in the moment sequence
\(\{\mathcal M_n\}_{n\ge0}\): the nodes \(\lambda_j\), the weights \(w_j\), and
the number of exponential terms correspond respectively to the source
locations, strengths, and source number.

To extract the number of exponential terms from the moment sequence, we define,
for \(m\ge1\), the Hankel matrix
\begin{align}
H_m=\left(\mathcal M_{r+c}\right)_{r,c=0}^{m-1}.
\end{align}
Using the finite exponential-sum representation \eqref{sec:eqf}, we obtain the factorization
\begin{align}
H_m=V_m(\lambda)\operatorname{diag}(w_1,\ldots,w_N)V_m(\lambda)^{\top},
\end{align}
where
\begin{align}
V_m(\lambda)=\left(\lambda_j^r\right)_{r=0,\ldots, m-1;\ j=1,\ldots,N}.
\end{align}

\begin{theorem}
\label{thm:hankel-rank}
Assume that the source locations are distinct and that the source strengths
are nonzero. Equivalently,
\[
\lambda_i\neq\lambda_j\quad (i\neq j),
\qquad
w_j\neq0.
\]
Then, for every \(m\ge N\),
\[
\operatorname{rank}H_m=N.
\]
Consequently, if \(m\) is chosen no smaller than the true source number, then
the number of point sources is determined by the rank of the Hankel matrix.
\end{theorem}

\begin{proof}
Since the nodes \(\lambda_1,\ldots,\lambda_N\) are pairwise distinct,
\(V_m(\lambda)\) has full column rank for \(m\ge N\). Since all weights are
nonzero, \(\operatorname{diag}(w_1,\ldots,w_N)\) is nonsingular. Therefore the
factorization of \(H_m\) gives \(\operatorname{rank}H_m=N\).
\end{proof}
\begin{remark}
Theorem~\ref{thm:hankel-rank} is the algebraic basis of the proposed
source-counting method. It shows that the low-frequency boundary moments
convert the inverse point-source counting problem into a finite-dimensional
rank-identification problem. In contrast to complex-frequency determinant
methods, no contour integral or zero-order computation is needed at this stage:
the number of sources is read directly from the rank of a static Hankel matrix.

In exact data, this gives \(N=\operatorname{rank}H_m\) for any \(m\geq N\).
For discrete and noisy data, the exact rank is replaced by a numerical rank
computed from the singular values of the empirical Hankel matrix. The method is
therefore stable when the \(N\)-th singular value is well separated from the
noise floor, which in turn requires sufficiently separated source nodes and
non-negligible source strengths.
\end{remark}


\section{Discrete Data and Numerical Rank Selection}
\label{sec:discrete-noisy-rank}

This section turns the exact rank characterization into a practical
source-counting procedure from discrete and noisy boundary data. In
computations, the low-frequency boundary data \(g(\theta)\) are available only
at finitely many boundary points and are usually contaminated by measurement
noise. We first approximate the Fourier moments by a discrete quadrature and
form an empirical Hankel matrix. We then estimate how boundary sampling and
measurement noise perturb the moments. Finally, using Weyl's inequality for
singular values, we derive a sufficient condition under which the numerical
rank of the perturbed Hankel matrix still gives the true source number.

\subsection{Discrete and noisy moments}

Suppose that the boundary is sampled at uniformly distributed points
\[
\theta_\ell=\frac{2\pi(\ell-1)}{L},
\qquad
\ell=1,\ldots,L.
\]
The positive Fourier moments are approximated by the trapezoidal rule
\begin{align}
\label{eq:discrete-moment}
\mathcal M_n^{(L)}=\frac1L\sum_{\ell=1}^{L}g(\theta_\ell)e^{-in\theta_\ell},
\qquad
n=0,1,\ldots,2m-2.
\end{align}
The corresponding discrete Hankel matrix is
\begin{align}
\label{eq:discrete-hankel-rank}
H_m^{(L)}=\left(\mathcal M_{r+c}^{(L)}\right)_{r, c=0}^{m-1}.
\end{align}

In the presence of noisy boundary data, let
\begin{align}
\label{eq:noisy-boundary-lowfreq}
g_\ell^\delta=g(\theta_\ell)+\eta_\ell,
\qquad
\ell=1,\ldots,L,
\end{align}
where \(\eta_\ell\) denotes the measurement noise at the boundary point
\(\theta_\ell\). The noisy discrete moments are
\begin{align}
\label{eq:noisy-discrete-moment-rank}
\widetilde{\mathcal M}_n
=
\frac1L
\sum_{\ell=1}^{L}
g_\ell^\delta e^{-in\theta_\ell},
\qquad
n=0,1,\ldots,2m-2,
\end{align}
and the corresponding empirical Hankel matrix is
\begin{align}
\label{eq:noisy-hankel-rank}
\widetilde H_m
=
\left(\widetilde{\mathcal M}_{r+c}\right)_{r,c=0}^{m-1}.
\end{align}
The exact rank identity in Theorem~\ref{thm:hankel-rank} no longer applies
directly to \(\widetilde H_m\), because noise and quadrature errors lift the
zero singular values of the ideal Hankel matrix. Therefore the exact rank must
be replaced by a numerical rank determined from the singular-value spectrum.

\subsection{A bound for the moment error}
\label{subsec:moment-error}

We first relate the perturbation of the moment sequence to boundary noise and
discrete sampling. Write
\[
\widetilde{\mathcal M}_n
=
\mathcal M_n+e_n,
\qquad
n=0,1,\ldots,2m-2.
\]
We shall use the aggregate moment error
\begin{align}
\label{eq:moment-error-bound}
\max_{0\leq n\leq 2m-2}|e_n|
\leq
\varepsilon_m .
\end{align}

Assume first that the boundary noise satisfies the deterministic bound
\[
|\eta_\ell|\leq\delta_g,
\qquad
\ell=1,\ldots,L.
\]
Then the noise contribution to the moment perturbation is bounded by
\begin{align}
\left|
\frac1L
\sum_{\ell=1}^{L}
\eta_\ell e^{-in\theta_\ell}
\right|
\leq
\delta_g,
\qquad
n=0,1,\ldots,2m-2.
\end{align}

There is also a deterministic quadrature error. Since all sources lie strictly
inside the unit disk, the Fourier coefficients of \(g\) decay geometrically.
Indeed, by  \eqref{sec2:eqg} and \eqref{sec2:eqgr}, we have
\[
g(\theta)
=
-\frac1{2\pi}
\sum_{j=1}^{N}q_j
\sum_{k\in\mathbb Z}
\rho_j^{|k|}e^{ik(\theta-\phi_j)}.
\]
Denoting $c_k:=
-\frac1{2\pi}
\sum_{j=1}^{N}
q_j\rho_j^{|k|}e^{-ik\phi_j}$,
we have
\[
g(\theta)=\sum_{k\in\mathbb Z}c_k e^{ik\theta}.
\]
Let
\[
\rho_*:=\max_{1\leq j\leq N}|p_j|<1,
\qquad
Q:=\frac1{2\pi}\sum_{j=1}^{N}|q_j|.
\]
It follows that
\[
|c_k|
\leq
\frac1{2\pi}
\sum_{j=1}^{N}|q_j|\rho_j^{|k|}
\leq
Q\rho_*^{|k|},
\qquad k\in\mathbb Z.
\]
For \(L>2m-2\), the trapezoidal rule gives the aliasing relation
\[
\mathcal M_n^{(L)}
=
\sum_{j\in\mathbb Z}c_{n+jL},
\qquad
0\leq n\leq2m-2.
\]
Hence we get
\begin{align}
\label{eq:aliasing-error-bound}
\left|
\mathcal M_n^{(L)}-\mathcal M_n
\right|
&\leq
\sum_{j\neq0}|c_{n+jL}|\leq
\frac{Q\bigl(\rho_*^{L-n}+\rho_*^{L+n}\bigr)}
{1-\rho_*^L}                                                   \notag\\
&\leq
\frac{2Q\rho_*^{L-(2m-2)}}{1-\rho_*^L},
\qquad
0\leq n\leq2m-2.
\end{align}
Consequently, the total moment error can be bounded by
\begin{align}
\label{eq:epsilon-m-bound}
\varepsilon_m
\leq
\delta_g
+
\frac{2Q\rho_*^{L-(2m-2)}}{1-\rho_*^L}.
\end{align}
This estimate separates the measurement-noise contribution from the
deterministic boundary-sampling error.

\subsection{Stability of numerical rank selection}
\label{subsec:rank-stability}


We now show how the moment perturbation affects the singular values of the
Hankel matrix. Define
\[
E_m
:=
\widetilde H_m-H_m
=
(e_{r+c})_{r,c=0}^{m-1}.
\]
By \eqref{eq:moment-error-bound}, each entry of \(E_m\) is bounded by
\(\varepsilon_m\). Hence, with \(\|\cdot\|_2\) denoting the spectral norm and
\(\|\cdot\|_F\) denoting the Frobenius norm,  we have
\begin{align}
\label{eq:hankel-perturb-bound}
\|E_m\|_2
\leq
\|E_m\|_F
=
\left(
\sum_{r,c=0}^{m-1}|e_{r+c}|^2
\right)^{1/2}
\leq
m\varepsilon_m .
\end{align}

\begin{lemma}
\label{lem:weyl-hankel}
Let
\[
\sigma_1(H_m)\geq\sigma_2(H_m)\geq\cdots\geq\sigma_m(H_m)\geq0
\]
and
\[
\sigma_1(\widetilde H_m)\geq\sigma_2(\widetilde H_m)\geq\cdots
\geq\sigma_m(\widetilde H_m)\geq0
\]
be the singular values of \(H_m\) and \(\widetilde H_m\), respectively.
If \eqref{eq:moment-error-bound} holds, then
\begin{align}
\label{eq:weyl-singular-bound}
\left|
\sigma_k(\widetilde H_m)-\sigma_k(H_m)
\right|
\leq
m\varepsilon_m,
\qquad
k=1,\ldots,m.
\end{align}
\end{lemma}

\begin{proof}
By Weyl's perturbation inequality for singular values, we have
\[
\left|
\sigma_k(\widetilde H_m)-\sigma_k(H_m)
\right|
\leq
\|\widetilde H_m-H_m\|_2
=
\|E_m\|_2 .
\]
Combining this estimate with \eqref{eq:hankel-perturb-bound} gives
\eqref{eq:weyl-singular-bound}.
\end{proof}

The preceding lemma yields a sufficient condition for stable source counting.
Let \(m\geq N\). By Theorem~\ref{thm:hankel-rank},
\[
\operatorname{rank}H_m=N.
\]
Therefore
\[
\sigma_N(H_m)>0,
\qquad
\sigma_{N+1}(H_m)=\cdots=\sigma_m(H_m)=0.
\]
The quantity \(\sigma_N(H_m)\) is the smallest signal singular value of the
ideal Hankel matrix.

\begin{theorem}
\label{thm:stable-rank-count}
Assume that \(m\geq N\), the source locations are distinct, and the source
strengths are nonzero. Suppose that the moment perturbation satisfies
\eqref{eq:moment-error-bound} and
\begin{align}
\label{eq:rank-stability-condition}
m\varepsilon_m
<
\frac12\sigma_N(H_m).
\end{align}
Let the numerical rank estimator be
\begin{align}
\label{eq:numerical-rank-estimator}
\widehat N
=
\#\{k:\sigma_k(\widetilde H_m)>\tau_m\},
\end{align}
where the threshold \(\tau_m\) satisfies
\begin{align}
\label{eq:threshold-admissible}
m\varepsilon_m
<
\tau_m
<
\sigma_N(H_m)-m\varepsilon_m .
\end{align}
Then
\begin{align}
\label{eq:stable-rank-result}
\widehat N=N.
\end{align}
\end{theorem}

\begin{proof}
For \(k\leq N\), Lemma~\ref{lem:weyl-hankel} gives
\[
\sigma_k(\widetilde H_m)
\geq
\sigma_k(H_m)-m\varepsilon_m
\geq
\sigma_N(H_m)-m\varepsilon_m
>
\tau_m .
\]
For \(k>N\), since \(\sigma_k(H_m)=0\), Lemma~\ref{lem:weyl-hankel} gives
\[
\sigma_k(\widetilde H_m)
\leq
m\varepsilon_m
<
\tau_m .
\]
Thus exactly the first \(N\) singular values of \(\widetilde H_m\) exceed
\(\tau_m\), which proves \(\widehat N=N\).
\end{proof}

\begin{remark}
Theorem~\ref{thm:stable-rank-count} shows that the stability of the
moment--Hankel rank method is governed by the gap between the smallest signal
singular value \(\sigma_N(H_m)\) and the zero singular values of the ideal
Hankel matrix. Since \(\sigma_{N+1}(H_m)=0\), this gap is precisely
\(\sigma_N(H_m)\). Close source locations or very weak source strengths reduce
\(\sigma_N(H_m)\), and therefore reduce the admissible perturbation level.
This explains why near-colliding sources and weak sources are more difficult
to identify from noisy data.
\end{remark}

%
%

\subsection{Practical rank estimators}
\label{subsec:practical-rank-estimators}

Theorem~\ref{thm:stable-rank-count} gives a sufficient condition for exact
rank recovery. Its threshold condition is useful for analysis, but it cannot
be applied directly in computation because it involves the unknown ideal
quantity \(\sigma_N(H_m)\) and the unknown source number \(N\). We therefore
use practical rank-selection rules based only on the singular values of the
empirical Hankel matrix \(\widetilde H_m\).

Let
\[
\sigma_1^\delta\geq\sigma_2^\delta\geq\cdots\geq\sigma_m^\delta\geq0
\]
be the singular values of \(\widetilde H_m\). In the ideal case, exactly the
first \(N\) singular values are nonzero, while the remaining singular values
are zero. In noisy data, the zero singular values are lifted by perturbations.
Thus the source-counting problem becomes the problem of separating significant
singular values from the noise floor.

When an estimate of the relative moment noise level \(\delta\) is available,
we use the noise-aware threshold
\begin{align}
\label{eq:noise-aware-threshold}
\tau_\delta
=
c_{\rm noise}\,\delta\,\sigma_1^\delta,
\end{align}
where \(c_{\rm noise}>0\) is a safety factor and \(\sigma_1^\delta\) gives the
scale of the empirical Hankel matrix. The estimated source number is
\begin{align}
\label{eq:noise-aware-rank-estimator}
\widehat N
=
\#\{k:\sigma_k^\delta>\tau_\delta\}.
\end{align}
This rule counts the singular values that are significantly above the estimated
noise level.

When the noise level is unavailable, one may instead use the largest relative
gap in the singular-value sequence:
\[
\widehat N
=
\argmax_{1\leq k\leq m-1}
\frac{\sigma_k^\delta}{\sigma_{k+1}^\delta}.
\]
This heuristic is motivated by the ideal rank structure: the largest drop in
the singular values is expected to occur between the last signal singular value
\(\sigma_N^\delta\) and the first noise-dominated singular value
\(\sigma_{N+1}^\delta\).

The numerical experiments below compare these two practical rank-selection
rules. Their performance is then interpreted using the stability ratio
\[
\frac{m\varepsilon_m}{\sigma_N(H_m)},
\]
which measures the perturbation level relative to the smallest ideal signal
singular value.



\section{Numerical Experiments}
\label{sec:numerical}

In this section, we test the proposed moment--Hankel rank method and examine
its stability in connection with the singular-value perturbation result in
Theorem~\ref{thm:stable-rank-count}. The purpose is to illustrate the
low-frequency rank characterization established in Theorem~\ref{thm:hankel-rank}
and to show how the smallest signal singular value controls the practical
performance of numerical rank selection.

All experiments are carried out in the unit disk. The low-frequency moments are
generated from the finite exponential-sum representation
\[
\mathcal M_n
=
\sum_{j=1}^{N}w_j\lambda_j^n,
\qquad
\lambda_j=\overline{p_j},
\qquad
w_j=-\frac{q_j}{2\pi},
\]
derived in Section~\ref{sec:hankel-rank}. Unless otherwise stated, the Hankel
order is chosen as \(m=8\).

For noisy experiments, we perturb the moment sequence directly in order to
isolate the stability of the moment--Hankel rank step. For a relative moment
noise level \(\delta\), we use
\[
\mathcal M_n^\delta
=
\mathcal M_n
+
\delta
\left(
\frac1{2m-1}\sum_{k=0}^{2m-2}|\mathcal M_k|^2
\right)^{1/2}
\xi_n,
\qquad
n=0,1,\ldots,2m-2,
\]
where \(\xi_n\) are independent standard complex Gaussian random variables. The
source number is estimated either by the noise-aware threshold rule
\[
\widehat N
=
\#\{k:\sigma_k^\delta>\tau_\delta\},
\qquad
\tau_\delta=c_{\rm noise}\delta\sigma_1^\delta,
\]
or by the largest-gap rule applied to the singular-value sequence of the noisy
Hankel matrix. In the reported experiments, \(c_{\rm noise}=5\), and each
Monte Carlo success rate is computed from 200 independent noise realizations.

For each noisy trial, we also compute the observed moment perturbation
\[
\varepsilon_m^{\rm obs}
=
\max_{0\leq n\leq 2m-2}
|\mathcal M_n^\delta-\mathcal M_n|
\]
and the stability ratio
\[
R
=
\frac{m\varepsilon_m^{\rm obs}}{\sigma_N(H_m)}.
\]
According to Theorem~\ref{thm:stable-rank-count}, the condition
\[
R<\frac12
\]
is sufficient for stable numerical rank recovery, provided that the rank
threshold is chosen between the perturbation level and the smallest signal
singular value.

\subsection{Noiseless Hankel rank pattern}

We first verify the rank characterization in the noiseless case. For
\(N=1,\ldots,5\), we construct deterministic source configurations with
distinct locations and nonzero strengths, compute the exact moment sequence,
and form the Hankel matrices \(H_m\) for \(m=1,\ldots,9\).

Figure~\ref{fig:noiseless-rank} shows the singular-value spectra of \(H_8\) and
the numerical rank as the Hankel order varies. The singular-value spectra
display a sharp drop after the \(N\)-th singular value, while the remaining
singular values are at the level of machine precision. The numerical rank
pattern is exactly
\[
\operatorname{rank}H_m=\min\{m,N\},
\]
and in particular
\[
\operatorname{rank}H_m=N
\qquad
\text{for all }m\geq N.
\]
This confirms the moment--Hankel rank principle in
Theorem~\ref{thm:hankel-rank}.

\begin{figure}[htbp]
\centering
\begin{subfigure}{0.48\textwidth}
\centering
\includegraphics[width=\textwidth]{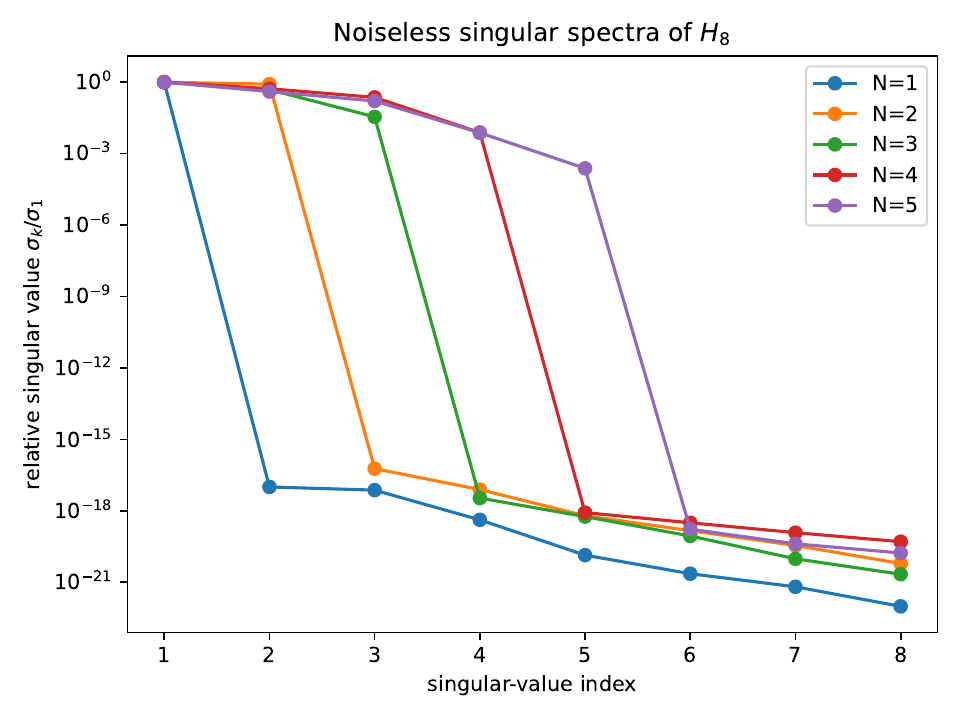}
\caption{Relative singular values of \(H_8\).}
\end{subfigure}
\hfill
\begin{subfigure}{0.48\textwidth}
\centering
\includegraphics[width=\textwidth]{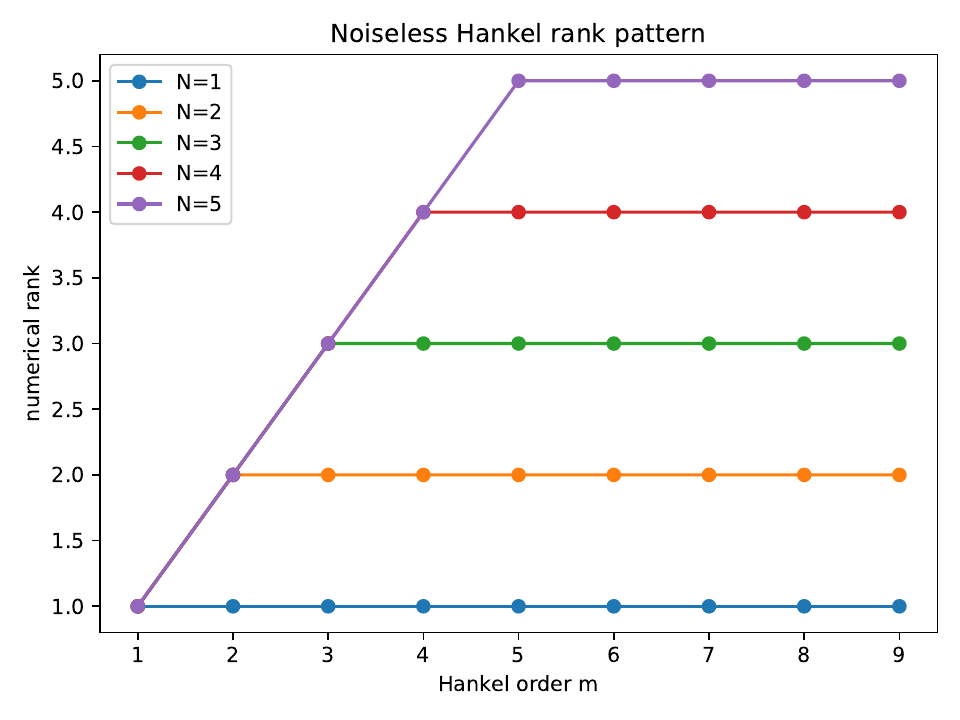}
\caption{Numerical rank as \(m\) varies.}
\end{subfigure}
\caption{Noiseless moment--Hankel rank identification. The singular-value
spectra show a clear drop after the \(N\)-th singular value, and the numerical
rank satisfies \(\operatorname{rank}H_m=N\) once \(m\geq N\).}
\label{fig:noiseless-rank}
\end{figure}

\subsection{Noise sensitivity and Weyl-type stability threshold}

We next examine the effect of moment noise on numerical rank selection and
compare the observed behavior with the Weyl-type stability condition in
Theorem~\ref{thm:stable-rank-count}. For each \(N=1,\ldots,5\), the same source
configurations as in the noiseless experiment are used. Independent complex
Gaussian perturbations are added to the moments at the noise levels
\[
\delta
=
0,\ 10^{-5},\ 10^{-4},\ 10^{-3},\ 10^{-2},\
3\times10^{-2},\ 5\times10^{-2}.
\]

Figure~\ref{fig:noise-rank-selection} compares the largest-gap rule and the
noise-aware threshold rule. The threshold rule is more consistent with the
singular-value perturbation analysis. For \(N=1\) and \(N=2\), it correctly
identifies the source number for all tested noise levels. For \(N=3\) and
\(N=4\), it remains accurate for small noise levels, but fails when the
smallest signal singular values are absorbed into the noise floor. For \(N=5\),
the rank estimate becomes sensitive much earlier. The largest-gap rule is
useful as a heuristic when no noise level is available, but it becomes less
stable for larger \(N\), especially when noise lifts several small singular
values.

The stability ratio shown in Figure~\ref{fig:weyl-stability} explains this
transition. When the median value of
\[
R=\frac{m\varepsilon_m^{\rm obs}}{\sigma_N(H_m)}
\]
stays below \(1/2\), the threshold rank selector is reliable. Once this ratio
crosses the sufficient threshold \(1/2\), the success rate drops sharply. Thus
the numerical transition from successful rank recovery to failure is consistent
with the sufficient condition in Theorem~\ref{thm:stable-rank-count}. The
condition is not necessary, but it provides a clear and computable diagnostic
for when numerical rank selection can be trusted.

The dominant failure mode is underestimation of the source number. This is
expected from the singular-value structure: noise lifts the zero singular
values, while small nonzero signal singular values may be pushed below the
chosen threshold. As a result, sources associated with weak singular directions
are missed.

\begin{figure}[htbp]
\centering
\begin{subfigure}{0.48\textwidth}
\centering
\includegraphics[width=\textwidth]{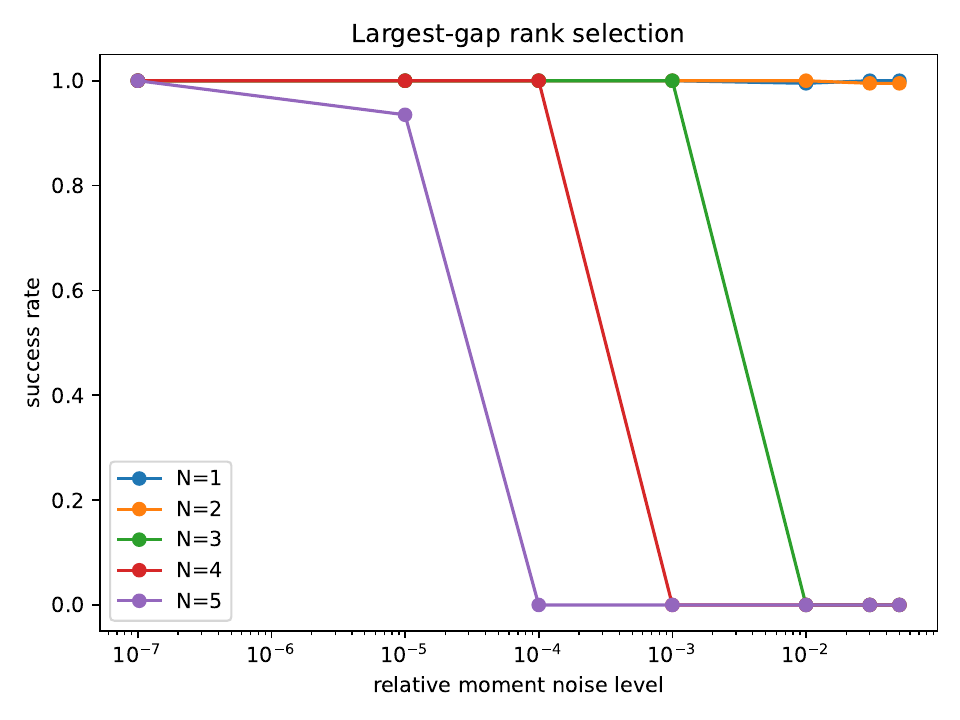}
\caption{Largest-gap rank selection.}
\end{subfigure}
\hfill
\begin{subfigure}{0.48\textwidth}
\centering
\includegraphics[width=\textwidth]{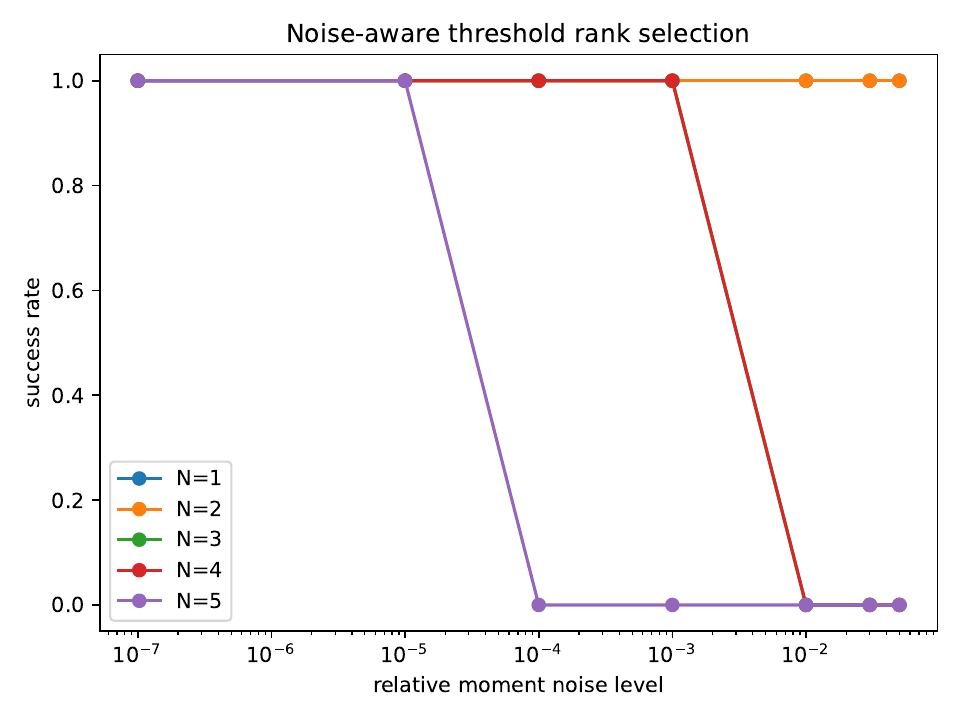}
\caption{Noise-aware threshold rank selection.}
\end{subfigure}
\caption{Monte Carlo success rates for source-number estimation under moment
noise. The noise-aware threshold rule is more consistent with the perturbation
analysis, while the largest-gap rule is more sensitive for larger source
numbers.}
\label{fig:noise-rank-selection}
\end{figure}

\begin{figure}[htbp]
\centering
\begin{subfigure}{0.48\textwidth}
\centering
\includegraphics[width=\textwidth]{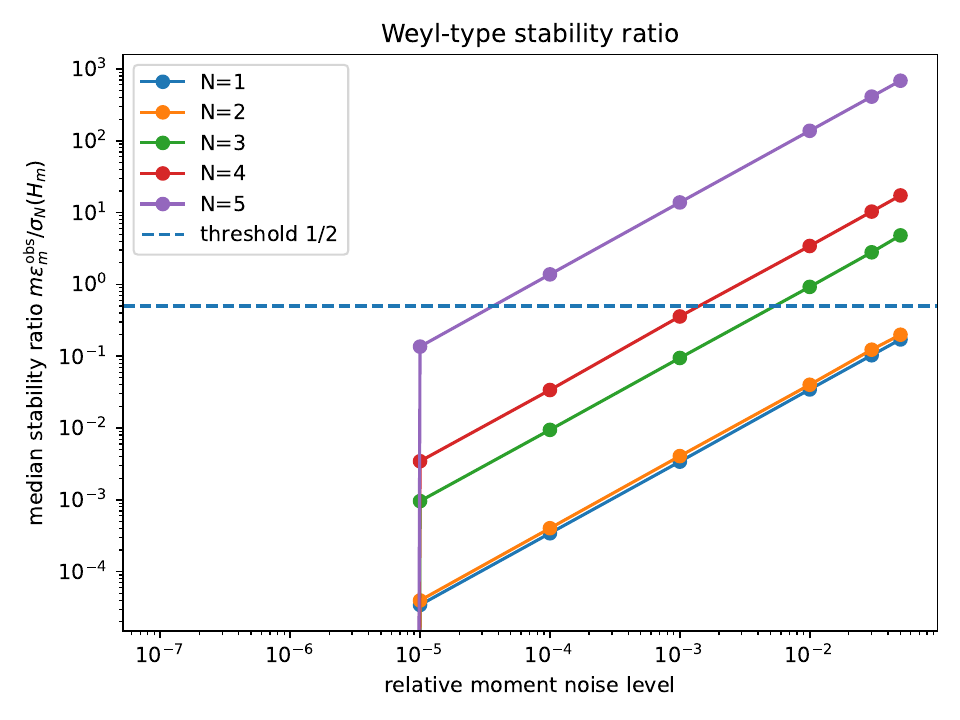}
\caption{Median stability ratio.}
\end{subfigure}
\hfill
\begin{subfigure}{0.48\textwidth}
\centering
\includegraphics[width=\textwidth]{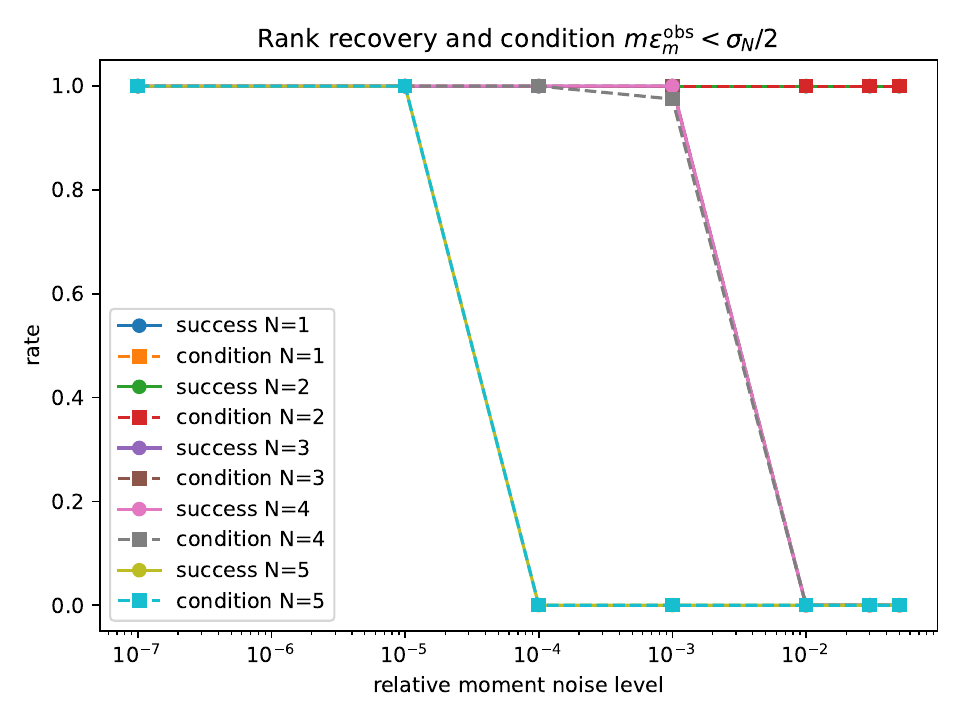}
\caption{Success rate and sufficient-condition rate.}
\end{subfigure}
\caption{Weyl-type stability diagnostics. The dashed line in panel (a)
indicates the sufficient threshold \(R=1/2\). The transition from successful
rank recovery to failure occurs when
\(m\varepsilon_m^{\rm obs}/\sigma_N(H_m)\) crosses this stability level.}
\label{fig:weyl-stability}
\end{figure}

\subsection{Resolution limits: close sources and weak sources}

The stability theorem shows that reliable numerical rank selection requires
\[
m\varepsilon_m < \frac12\sigma_N(H_m).
\]
Thus the admissible perturbation level is proportional to the smallest nonzero
singular value \(\sigma_N(H_m)\). We now examine two configurations in which
this singular value becomes small: close source locations and weak source
strengths.

First, we consider a two-source configuration and decrease the separation
\(|p_1-p_2|\). Figure~\ref{fig:resolution-limits}(a) plots the smallest signal
singular value \(\sigma_N(H_m)\) as a function of the source separation. As the
two sources approach each other, \(\sigma_N(H_m)\) decays rapidly. Although the
exact rank remains equal to \(2\), the numerical rank becomes increasingly
difficult to distinguish in noisy data because the admissible perturbation
level decreases.

Second, we consider a three-source configuration and gradually decrease the
strength \(q_3\) of the third source. Figure~\ref{fig:resolution-limits}(b)
shows that the smallest signal singular value decreases as the weak source
strength is reduced. Therefore weak sources can be missed by numerical rank
selection even though they are present in the exact moment sequence.

Figure~\ref{fig:admissible-eps} reports the admissible perturbation level
\[
\frac{\sigma_N(H_m)}{2m}
\]
predicted by Theorem~\ref{thm:stable-rank-count}. Both close sources and weak
sources reduce this admissible level. These results explain the sensitivity
observed in the noisy rank experiments: the rank method is stable when the
smallest signal singular value is clearly separated from the noise floor, and
it becomes unstable when this separation is lost.

\begin{figure}[htbp]
\centering
\begin{subfigure}{0.48\textwidth}
\centering
\includegraphics[width=\textwidth]{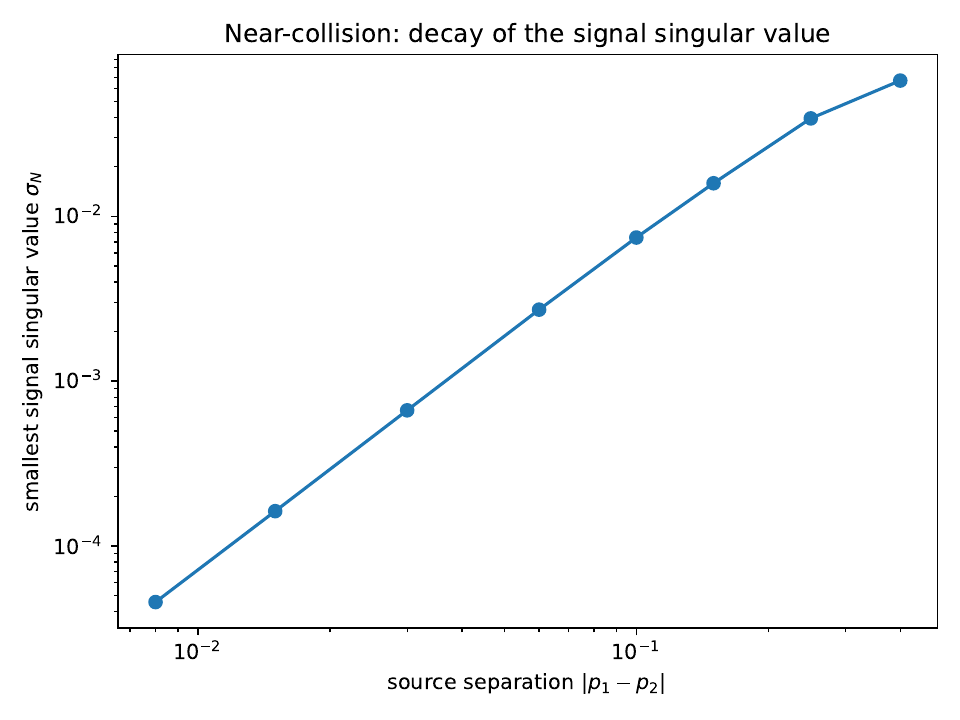}
\caption{Close-source experiment.}
\end{subfigure}
\hfill
\begin{subfigure}{0.48\textwidth}
\centering
\includegraphics[width=\textwidth]{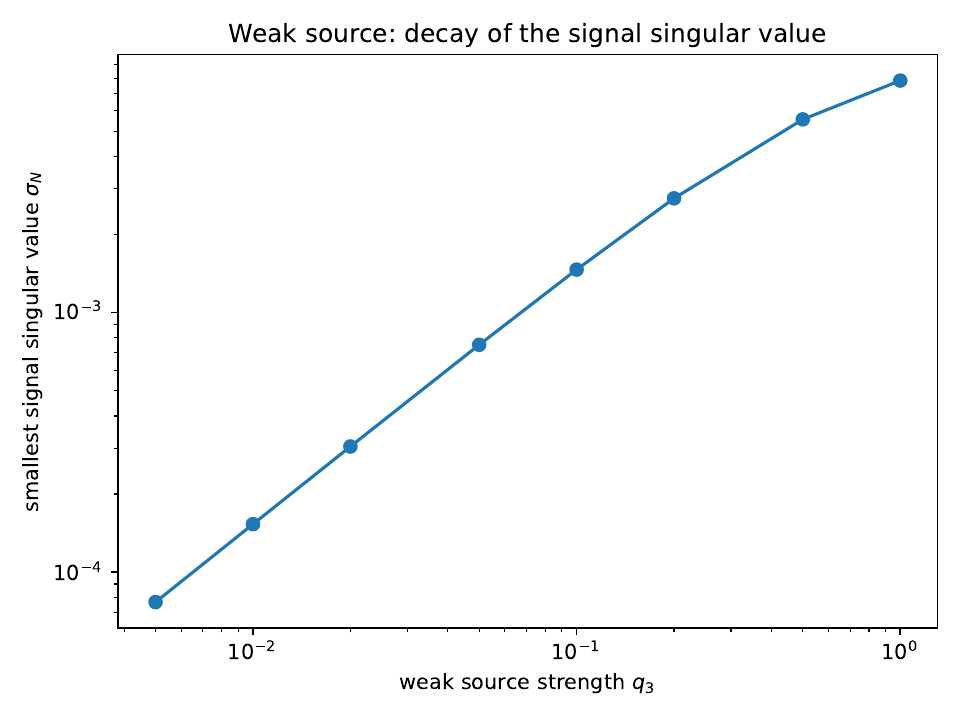}
\caption{Weak-source experiment.}
\end{subfigure}
\caption{Resolution limits of the moment--Hankel rank method. The smallest
signal singular value \(\sigma_N(H_m)\) decreases when sources become close or
when one source becomes weak.}
\label{fig:resolution-limits}
\end{figure}

\begin{figure}[htbp]
\centering
\begin{subfigure}{0.48\textwidth}
\centering
\includegraphics[width=\textwidth]{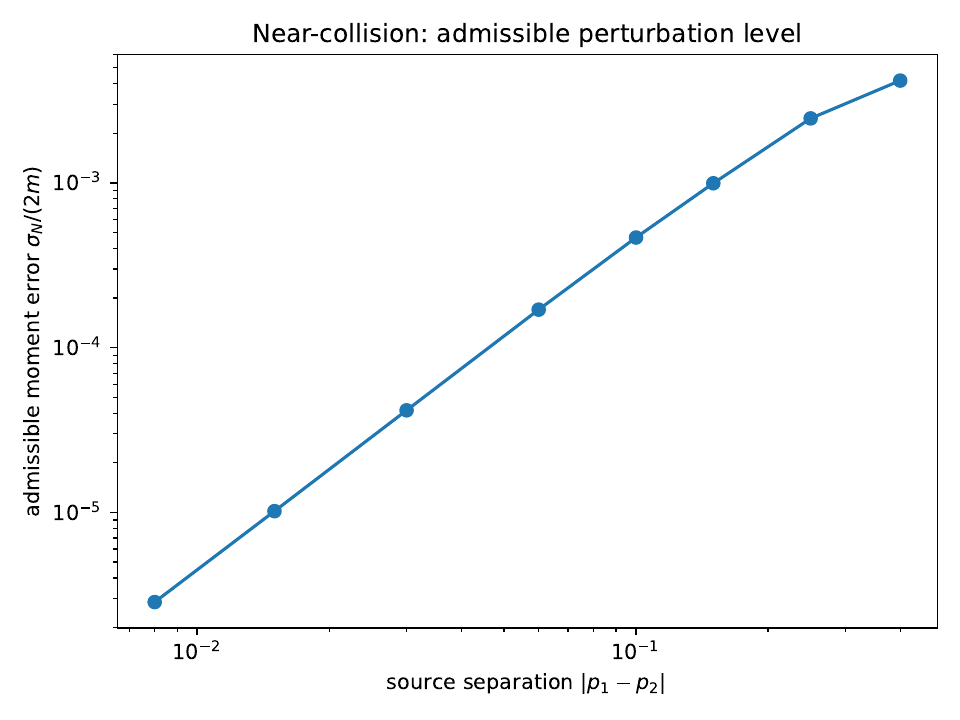}
\caption{Close-source experiment.}
\end{subfigure}
\hfill
\begin{subfigure}{0.48\textwidth}
\centering
\includegraphics[width=\textwidth]{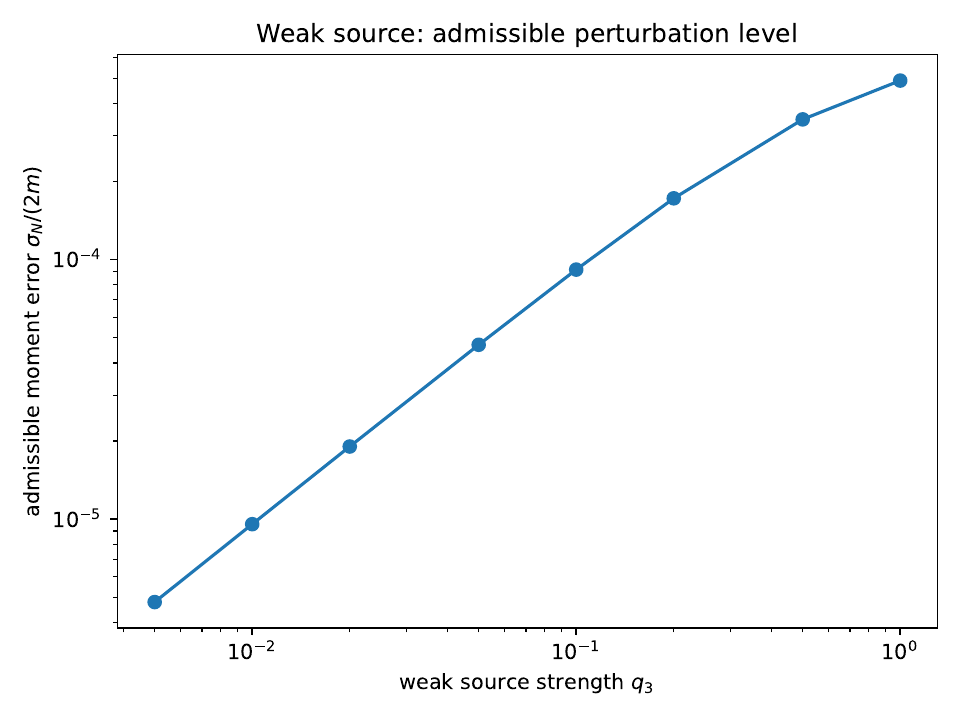}
\caption{Weak-source experiment.}
\end{subfigure}
\caption{Admissible moment perturbation level
\(\sigma_N(H_m)/(2m)\) predicted by the stability theorem. Close and weak
sources reduce the admissible perturbation level and therefore make numerical
rank selection more sensitive to noise.}
\label{fig:admissible-eps}
\end{figure}

\subsection{Location and strength recovery after rank identification}

After the source number has been identified, the same low-frequency moment
sequence can be used to recover source locations and strengths. In this
experiment, the true value of \(N\) is used in the reconstruction step in order
to isolate the stability of the Prony--Vandermonde recovery. The locations are
obtained from the roots of the annihilating polynomial, and the strengths are
then computed by solving the corresponding Vandermonde system.

We measure the location error by
\[
E_p
=
\left(
\frac1N
\sum_{j=1}^N
|\widehat p_{\pi(j)}-p_j|^2
\right)^{1/2},
\]
where the permutation \(\pi\) is chosen to minimize the matching error. The
relative strength error is defined by
\[
E_q
=
\frac{
\left(
\sum_{j=1}^N
|\widehat q_{\pi(j)}-q_j|^2
\right)^{1/2}
}{
\left(
\sum_{j=1}^N
|q_j|^2
\right)^{1/2}
}.
\]

Figure~\ref{fig:reconstruction} reports the reconstruction errors for
representative source numbers \(N=1,3,5\), together with the condition number
of the Vandermonde matrix. In all cases, the noiseless reconstruction reaches
machine precision. As the moment noise increases, both \(E_p\) and \(E_q\)
increase. The strength error is generally larger than the location error,
especially for larger \(N\), because the strengths are recovered from a
Vandermonde linear system after the nodes have already been estimated.

The Vandermonde condition number grows rapidly with \(N\). In the tested
configurations, it increases from \(1\) for \(N=1\) to approximately
\(1.49\times10^2\) for \(N=5\). This growth explains the deterioration of the
reconstruction for larger source numbers. Thus the recovery stage is accurate
for small and moderately conditioned configurations, but it becomes sensitive
when the Vandermonde system is ill conditioned.

\begin{figure}[htbp]
\centering
\begin{subfigure}{0.48\textwidth}
\centering
\includegraphics[width=\textwidth]{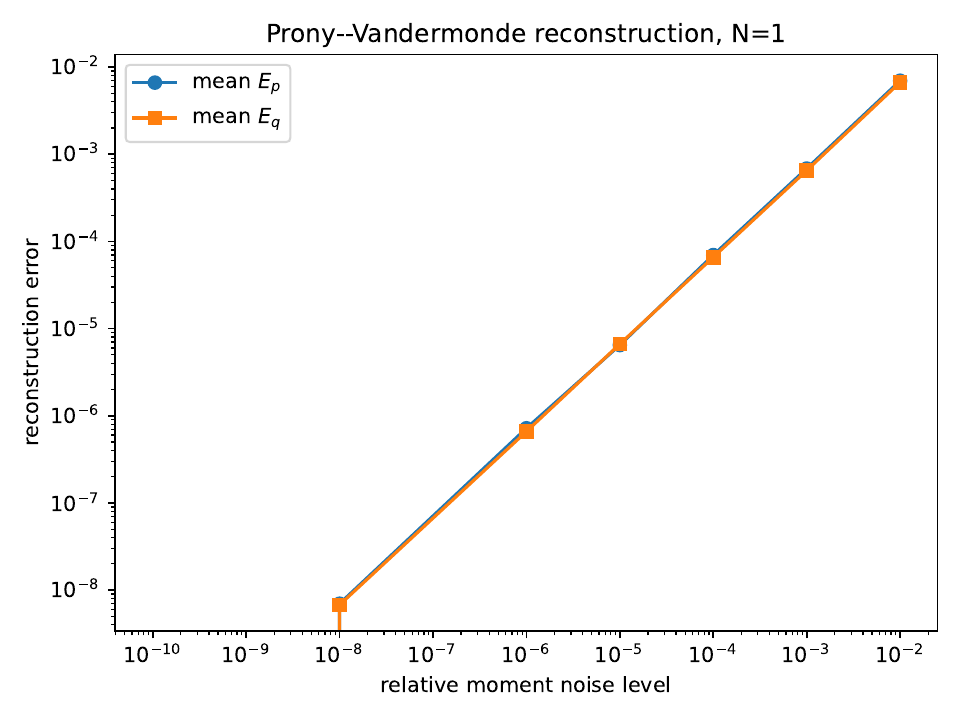}
\caption{\(N=1\).}
\end{subfigure}
\hfill
\begin{subfigure}{0.48\textwidth}
\centering
\includegraphics[width=\textwidth]{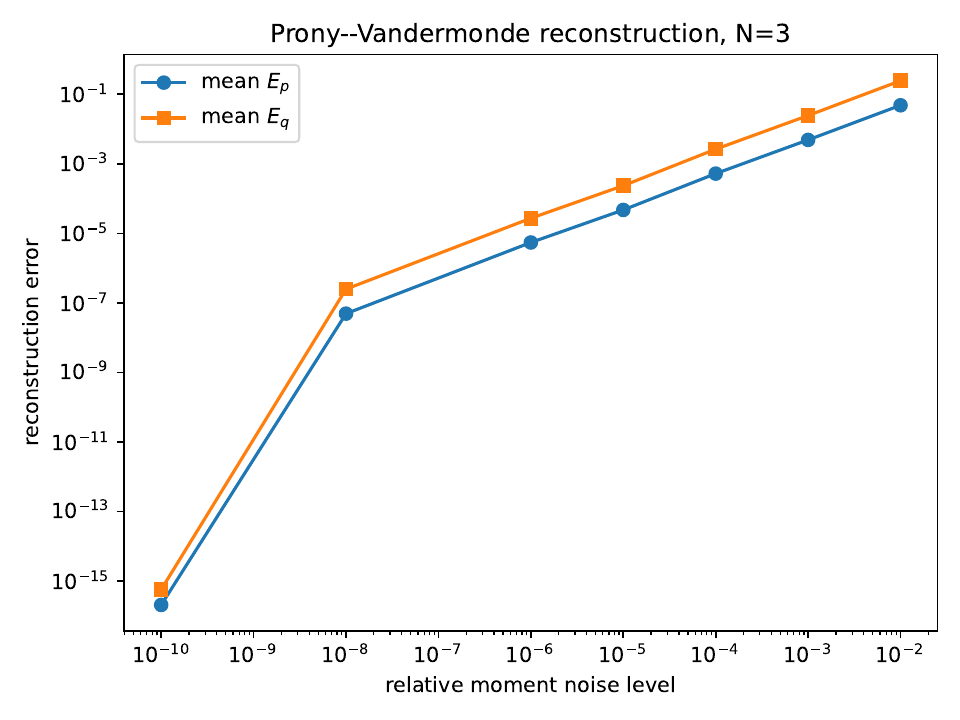}
\caption{\(N=3\).}
\end{subfigure}

\vspace{0.3cm}

\begin{subfigure}{0.48\textwidth}
\centering
\includegraphics[width=\textwidth]{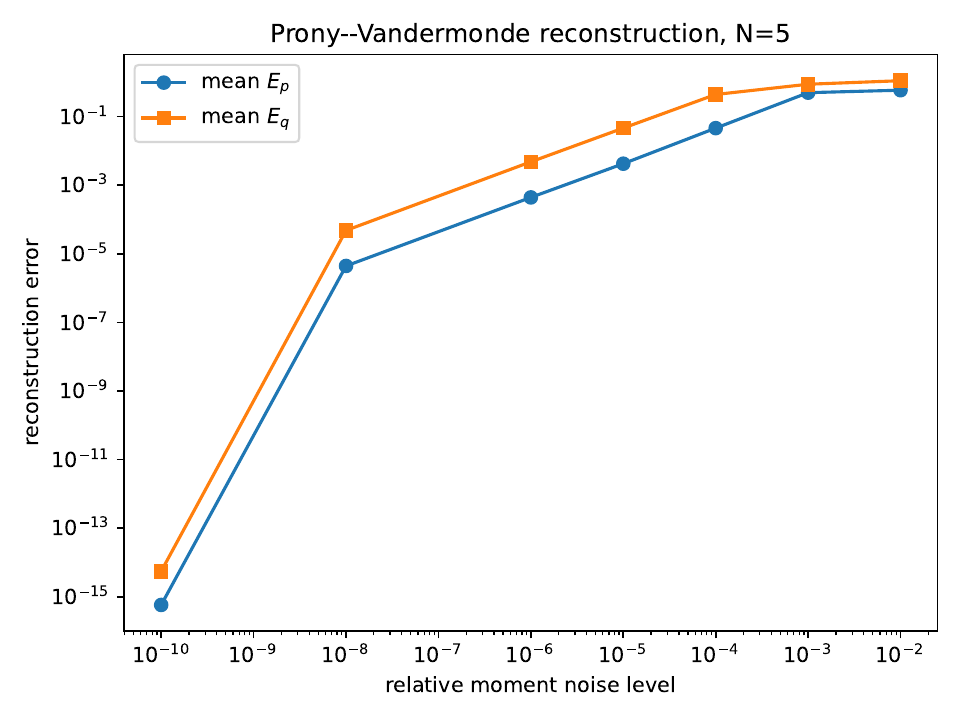}
\caption{\(N=5\).}
\end{subfigure}
\hfill
\begin{subfigure}{0.48\textwidth}
\centering
\includegraphics[width=\textwidth]{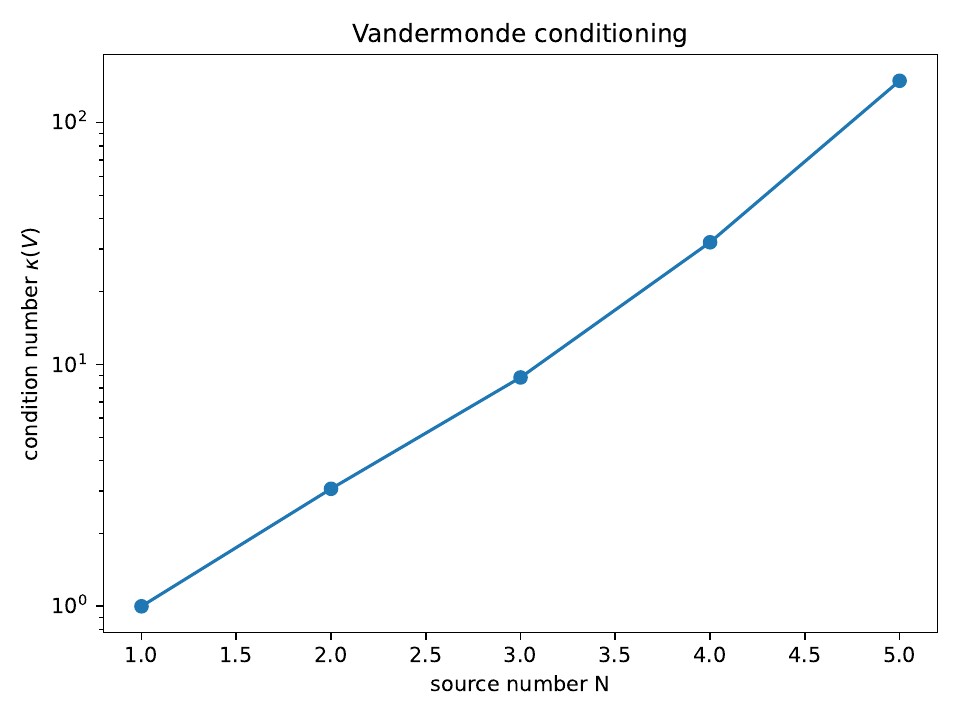}
\caption{Vandermonde conditioning.}
\end{subfigure}
\caption{Prony--Vandermonde reconstruction after source-number identification.
The reconstruction is accurate for small \(N\) and small moment noise, but
deteriorates as \(N\) increases because the associated Vandermonde system
becomes more ill conditioned.}
\label{fig:reconstruction}
\end{figure}

\subsection{Summary of the numerical findings}

The numerical experiments support the theoretical rank characterization and
clarify the practical limitations of the method. In the noiseless case, the
Hankel rank exactly recovers the source number once \(m\geq N\). Under moment
noise, the method remains effective when the stability ratio
\[
\frac{m\varepsilon_m^{\rm obs}}{\sigma_N(H_m)}
\]
is below the theoretical threshold \(1/2\). The main source of failure is rank
underestimation, which occurs for larger \(N\), nearly colliding sources, or
weak sources. These cases reduce \(\sigma_N(H_m)\), thereby reducing the
admissible perturbation level predicted by Theorem~\ref{thm:stable-rank-count}.

After the source number is correctly identified, the Prony--Vandermonde step
recovers locations and strengths accurately for well-conditioned
configurations. The strength recovery is more sensitive than the location
recovery because it requires solving a Vandermonde system. Overall, the
experiments show that the moment--Hankel rank method is an effective
low-frequency source-counting procedure when the significant singular values
are separated from the perturbation level.

\section{Conclusion}
\label{sec:conclusion}

This paper proposed a low-frequency moment--Hankel rank method for identifying
the number of stationary point sources in the heat equation in the unit disk.
After applying the Laplace transform and normalizing the boundary flux, we
considered the low-frequency limit of the resulting boundary data. The positive
Fourier moments of this limit were shown to have a finite exponential-sum
representation, in which the exponential nodes encode the source locations,
the weights encode the source strengths, and the number of exponential terms
equals the number of point sources.

The main result shows that the Hankel matrix generated by these moments admits
a Vandermonde factorization. Under the natural assumptions that the source
locations are distinct and the source strengths are nonzero, the rank of the
Hankel matrix is exactly the number of point sources, provided that the Hankel
order is no smaller than the true source number. This gives a direct
deterministic source-counting principle based on the low-frequency moment
sequence. Compared with determinant zero-counting methods in the complex
Laplace-frequency plane, the present approach avoids contour integration and
reduces the counting step to a finite-dimensional numerical rank problem.

We also discussed the implementation of the rank criterion from discrete and
noisy boundary data. In practical computation, the exact Hankel rank is
replaced by the numerical rank determined from the singular values of the
empirical Hankel matrix. The numerical experiments confirm the theoretical rank
pattern in the noiseless case and show that the method can reliably identify
the source number when the significant singular values are well separated from
the noise floor. The experiments further demonstrate that close sources and
weak sources reduce the smallest signal singular value, thereby making the
numerical rank selection more sensitive to noise.

After the source number is identified, the same moment sequence can be used to
recover source locations and strengths through an annihilating-polynomial and
Vandermonde reconstruction procedure. The numerical results show accurate
recovery for well-conditioned configurations, while also indicating that the
strength reconstruction is more sensitive than the location reconstruction due
to the conditioning of the Vandermonde system.

The present work focuses on the unit disk, where the low-frequency Fourier
moment representation is explicit. Extensions to more general domains,
development of more robust numerical rank-selection rules, and regularized
moment reconstruction under stronger noise are natural directions for future
work.







\bibliographystyle{plain}   %
\bibliography{myref}       %

@article{Abdelaziz2015,
	author = {B. Abdelaziz and A. El Badia and A. El Hajj},
	journal = {Inverse Problems},
	number = {105002},
	pages = {26pp},
	title = {Direct algorithms for solving some inverse source problems in 2D elliptic equations},
	volume = {31},
	year = {2015}}

@article{Deng2026_Hankel,
	author = {Z. Deng and A. Qian and X. Yang},
	journal = {https://arxiv.org/abs/2606.15065},
	title = {A {Hankel} determinant zero-order principle for source counting in an inverse heat point-source problem}}

@article{Vessella1992,
	author = {S. Vessella},
	journal = {Inverse Problems},
	number = {911--917},
	title = {Locations and strengths of point sources: stability estimates},
	volume = {8},
	year = {1992}}

@article{Ohe2011,
	author = {T. Ohe and H. Inui H and K. Ohnaka},
	journal = {Inverse Problems},
	number = {115011},
	title = {Real-time reconstruction of time-varying point sources in a three-dimensional scalar wave equation},
	volume = {27},
	year = {2011}}

@article{Mamonov2013,
	author = {A. V. Mamonov and Y-H R. Tsai},
	journal = {Inverse Problems},
	number = {035009},
	title = {Point source identification in nonlinear advection diffusion reaction systems},
	volume = {29},
	year = {2013}}

@article{Komornik2002,
	author = {V. Komornik and M. Yamamoto},
	journal = {Inverse Problems},
	pages = {319--329},
	title = {Upper and lower estimates in determining point sources in a wave equation},
	volume = {18},
	year = {2002}}

@article{Ren2019,
	author = {K. Ren and Y. Zhong},
	journal = {Inverse Problems},
	pages = {125003},
	title = {Imaging point sources in heterogeneous environments},
	volume = {35},
	year = {2019}}

@article{Ling2006,
	author = {L. Ling and M. Yamamoto and Y. C. Hon and T. Takeuchi},
	journal = {Inverse Problems},
	number = {4},
	pages = {1289},
	title = {Identification of source locations in two-dimensional heat equations},
	volume = {22},
	year = {2006}}

@article{Cannon1998,
	author = {J. R. Cannon and P. DuChateau},
	journal = {Inverse Problems},
	number = {3},
	pages = {535},
	title = {Structural identification of an unknown source term in a heat equation},
	volume = {14},
	year = {1998}}

@article{Kian2019,
	author = {Y. Kian and M. Yamamoto},
	journal = {Inverse Problems},
	number = {11},
	title = {Reconstruction and stable recovery of source terms and coefficients appearing in diffusion equations},
	volume = {35},
	year = {2019}}

@article{Gu2025,
	author = {Q. Gu and W. Zhang and Z. Zhang},
	journal = {SIAM Journal on Applied Mathematics},
	number = {5},
	pages = {2337-2354},
	title = {Determine the Point Source of the Heat Equation with Sparse Boundary Measurements},
	volume = {85},
	year = {2025}}

@article{Isakov1991,
	author = {V. Isakov},
	journal = {Communications on Pure and Applied Mathematics},
	number = {2},
	pages = {185-209},
	title = {Inverse parabolic problems with the final overdetermination},
	volume = {44},
	year = {1991}}

@article{Rundell1980,
	author = {W. Rundell and D. L. Colton},
	journal = {Applicable Analysis},
	number = {3},
	pages = {231-242},
	title = {Determination of an unknown non-homogeneous term in a linear partial differential equation from overspecified boundary data},
	volume = {10},
	year = {1980}}

@article{Bushuyev1995,
	author = {I. Bushuyev},
	journal = {Inverse Problems},
	number = {L11-L16},
	title = {Global uniqueness for inverse parabolic problems with final observation},
	volume = {11},
	year = {1995}}

@article{Choulli1994,
	author = {M. Choulli},
	journal = {Inverse Problems},
	pages = {1123--1132},
	title = {An inverse problem for a semilinear parabolic equation},
	volume = {10},
	year = {1994}}

@article{Reeve1994,
	author = {D. E. Reeve and M. Spivack},
	journal = {Inverse Problems},
	pages = {1335-1344},
	title = {Determination of a source term in the linear diffusion equation},
	volume = {10},
	year = {1994}}

@article{Solovev1989,
	author = {V. V. Solov'ev},
	journal = {Differential Equations},
	pages = {1114-1119},
	title = {On the solvability of the inverse problem of determining a source with overdetermination on the upper base for the parabolic equation},
	volume = {25},
	year = {1989}}

@article{Baratchart2005,
	author = {L. Baratchart and A. B. Abda and F. B. Hassen and J. Leblond},
	journal = {Inverse Problems},
	pages = {51--74},
	title = {Recovery of pointwise sources or small inclusions in 2d domains and rational approximation},
	volume = {21},
	year = {2005}}

@article{Kovalets2011,
	author = {I. V. Kovalets and S. Andronopoulos and A. G. Venetsanos and J. G. Bartzis},
	journal = {Mathematics and Computers in Simulation},
	pages = {244-257},
	title = {Identification of strength and location of stationary point source of atmospheric pollutant in urban conditions using computational fluid dynamics model},
	volume = {82},
	year = {2011}}

@article{Badia2005,
	author = {A. E. Badia and T. H. Duong and A. Hamdi},
	journal = {Inverse Problems},
	pages = {1121--1136},
	title = {Identification of a point source in a linear advection-dispersion-reaction equation: Application to a pollution source problem},
	volume = {21},
	year = {2005}}

@article{Deng2026,
	author = {Z. L. Deng and C. Li and X. M. Yang},
	journal = {https://arxiv.org/abs/2509.14245},
	title = {A {Bayesian} thinning algorithm for the point source identification of heat equation},
	year = {2025}}

@article{Gallet2022,
	author = {A. Gallet and S. Rigby and T. N. Tallman and X. Kong and I. Hajirasouliha and A. Liew and D. Liu and L. Chen and A. Hauptmann and D. Smyl},
	journal = {Proceedings of Royal Society A},
	number = {20210526},
	title = {Structural engineering from an inverse problems perspective},
	volume = {478},
	year = {2022}}

@article{He2018,
	author = {B. He and A. Sohrabpour and E. Brown and Z. Liu},
	journal = {Annual review of biomedical engineering},
	pages = {171-196},
	title = {Electrophysiological Source Imaging: A Noninvasive Window to Brain Dynamics},
	volume = {20},
	year = {2018}}

@article{Moghaddam2021,
	author = {M. B. Moghaddam and M. Mazaheri and J. M. V. Samani},
	journal = {Goundwater for sustainable development},
	number = {100651},
	title = {Inverse modeling of contaminant transport for pollution source identification in surface and groundwaters: a review},
	volume = {15},
	year = {2021}}

@article{Shlomi2007,
	author = {S. Shlomi and A. M. Michalak},
	journal = {Water resources research},
	number = {W03412},
	title = {A geostatistical framework for incorporating transport information in estimating the distribution of a groundwater contaminant plume},
	volume = {43},
	year = {2007}}

@article{Baillet2001,
	author = {S. Baillet and J. C. Mosher and R. M. Leahy},
	journal = {IEEE Signal Processing Magazine},
	number = {6},
	pages = {14-30},
	title = {Electromagnetic brain mapping},
	volume = {18},
	year = {2001}}

@article{Gong2026,
	archiveprefix = {arXiv},
	author = {F. Gong and B. Jin and Y. Kian and S. Liu},
	eprint = {2603.09248},
	journal = {arXiv preprint},
	primaryclass = {math.AP},
	title = {Identification of a Point Source in the Heat Equation from Sparse Boundary Measurements},
	year = {2026}}

@book{Isakov1990,
	address = {Providence, RI},
	author = {V. Isakov},
	publisher = {American Mathematical Society},
	title = {Inverse Source Problems},
	year = {1990}}

\end{document}